\documentclass[12pt,leqno]{amsart}
\allowdisplaybreaks[3]
\numberwithin{equation}{section}
\usepackage{url}
\usepackage{amssymb}
\usepackage{amsmath}
\usepackage{latexsym}

\def\R{\mathbb R}
\def\C{\mathbb C}

\def\N{\mathbb N}
\def\im{\operatorname{Im}}
\def\re{\operatorname{Re}}
\def\dim{\operatorname{dim}}

\def\dist{\operatorname{dist}}
\def\meas{\operatorname{meas}}

\def\diam{\operatorname{diam}}
\def\dens{\operatorname{dens}}
\def\area{\operatorname{area}}
\def\upperlogdens{\operatorname{\overline{log\;dens}}}

\newtheorem{lemma}{Lemma}[section]
\newtheorem{theorem}{Theorem}[section]

\theoremstyle{definition}

\theoremstyle{remark}

\newtheorem*{remark}{Remark}

\newtheorem{ack}{Acknowledgement}

\title[Hausdorff dimension of Julia sets]{On
 the Hausdorff dimension of the Julia set of a regularly growing
entire function}
\subjclass{37F10 (primary), 30D05, 30D15 (secondary)}
\thanks{The authors were supported by the EU Research Training
Network CODY. The first author was also supported by the 
the Deutsche Forschungsgemeinschaft, Be 1508/7-1,
the ESF Research Networking Programme
HCAA and the G.I.F., the
German--Israeli Foundation for Scientific Research and Development,
Grant G-809-234.6/2003.
The second author was also supported by 
Polish MNiSW Grant N N201 0222 33 and PW Grant 504G 1120 0011 000.}
\author{Walter Bergweiler}
\address{Mathematisches Seminar,
Christian--Albrechts--Universit\"at zu Kiel,
Lude\-wig--Meyn--Str.~4, D--24098 Kiel, Germany}
\email{bergweiler@math.uni-kiel.de}
\author{Bogus\l awa Karpi\'nska}
\address{Faculty of Mathematics and Information Science,
Warsaw University of Technology,
Pl.\ Politechniki 1, 00-661 Warszawa, Poland}
\email{bkarpin@mini.pw.edu.pl}
\date{\today}

\begin{document}
\begin{abstract}
We show that if the growth of a transcendental entire function $f$
is sufficiently regular, then the Julia set and the escaping set
of $f$ have Hausdorff dimension~$2$.
\end{abstract}

\maketitle
\section{Introduction and results}
The Julia set $J(f)$ of an entire function $f$ is defined as
the set of all
points in the plane where the iterates $f^n$ of $f$ do not form a
normal family.
Denote by $\dim E$ the Hausdorff dimension
and by $\area E$ the Lebesgue measure of a subset~$E$ of the plane.

McMullen ~\cite{McMullen87} proved that $\dim J(\lambda e^z)=2$
for $\lambda\in \C$, $\lambda\neq 0$.
He also proved that $\area J(\sin (\alpha z+\beta))>0$
and hence
$\dim J(\sin (\alpha z+\beta))=2$
for $\alpha, \beta, \in \C$, $\alpha \neq 0$.
In the proofs, he first showed that these results hold if  
the Julia set $J(f)$ is replaced by the escaping set 
$I(f)=\{z\in\C: f^n(z)\to\infty\}$, and then he showed that
$I(f)\subset J(f)$ for the functions $f$ considered.

McMullen's results have been extended to various classes of
entire functions; see~\cite{Aspenberg09,Baranski08,Bergweiler08a,
Schubert07,Taniguchi03}. 
All these extensions concern the 
Eremenko-Lyubich class $B$ which consists of
all entire functions for which the set of finite asymptotic values
and critical values is bounded. 
Here we only mention the result of Bara\'nski~\cite{Baranski08}
and Schubert~\cite{Schubert07} which says that 
$\dim J(f)=2$ if $f\in B$ and if $f$ has finite order.
Recall that the order $\rho(f)$ of an entire function $f$ is defined
by
\[
\rho(f)=\limsup_{r\to\infty}\frac{\log \log M(r,f)}{\log r},
\quad
\text{where}
\quad
M(r,f) =\max_{|z|=r} |f(z)| .
\]

One advantage of working with the class $B$ is that
$I(f)\subset J(f)$ for $f\in B$ by a result of 
Eremenko and Lyubich~\cite[Theorem~1]{Eremenko92} so that the second
part of McMullen's argument carries over directly to this class.

Eremenko and Lyubich prove their result that $I(f)\subset J(f)$
by introducing a logarithmic change of variable to the subject.
This logarithmic change of variable
has become a very powerful tool in transcendental dynamics 
and it is the main reason why
a considerable amount of research
has been devoted to the class~$B$. This includes results on
Hausdorff dimension (e.g.,~\cite{Baranski08a,Stallard96}),
but also on various other topics (e.g.,~\cite{Rempe,RRRS}).

The purpose of this paper is to obtain a result on the Hausdorff
dimension of the Julia set for entire functions which do not
belong to the Eremenko-Lyubich class and for which the 
logarithmic change of variable therefore is not available.
We consider functions which grow regularly in a certain sense.
More precisely, 
we will be concerned with entire functions $f$ 
for which there exist $A,B,C,r_0>1$ such that
\begin{equation}\label{regular}
A\log M(r,f)\leq \log M(Cr,f)\leq B\log M(r,f)
\quad \text{for}\ r\geq r_0.
\end{equation}
Our main result is the following.
\begin{theorem} \label{theorem2}
Let $f$ be an entire function satisfying~\eqref{regular}.
Then $\dim( I(f)\cap J(f))=2$.
\end{theorem}
We note that the hypothesis~\eqref{regular} is satisfied if 
there exists $c_1,c_2,\rho>0$ such that
\begin{equation}\label{veryreg}
c_1r^\rho \leq \log M(r,f) \leq c_2 r^\rho
\end{equation}
for large $r$ and thus in particular if there exists $c,\rho>0$
such that
\begin{equation}\label{complreg}
\log M(r,f) \sim c r^\rho
\end{equation}
as $r\to\infty$. It is classical 
that~\eqref{complreg} holds for transcendental entire functions
which satisfy 
an algebraic differential equation of first 
order~\cite[Section~IV.6]{Valiron23} or
a linear differential equation whose coefficients are rational 
functions~\cite[Section~IV.5]{Valiron23}.
As another example we mention Poincar\'e functions 
associated to repelling fixed points of polynomials or,
more generally,
transcendental entire solutions of the functional equation
$f(sz)=P(f(z),z)$
where $|s|>1$ and $P$ is a polynomial in two variables
with $\deg_f P\geq 2$.
A solution $f$ of such an
equation satisfies~\eqref{veryreg} for large~$r$;
see~\cite[Section~II.8]{Valiron23}.
Finally we note that~\eqref{complreg} is satisfied by
functions of completely regular growth in the sense
of Pfluger; see~\cite[Section~3]{Levin64} for a thorough
treatment of this class of functions.

We note that the condition~\eqref{regular} does not imply that
$I(f)\subset J(f)$. For example, for the function $f(z)=z+1+e^{-z}$
already considered by Fatou~\cite[Exemple~1, p.~358]{Fatou26}
we have $\log M(r,f)\sim r$ while 
$\{z\in\C:\re z>0\}\subset I(f)\setminus J(f)$.
We will further discuss the condition~\eqref{regular} in section~\ref{discreg}.

Among the tools used in the proof of Theorem~\ref{theorem2}
are the Ahlfors islands theorem (see 
Lemma~\ref{ahlfors} below) and a result on the
Hausdorff dimension of the intersection of nested sets due to
McMullen (see Lemma~\ref{lemmamcm} below). In addition,
the proof requires some careful estimates of the
logarithmic derivative of~$f$.
As these estimates of the logarithmic derivative may be of independent
interest, we include them in this introductory section.

For $\alpha_1,\alpha_2,q,\lambda\geq0$ we consider the set $T(f,\alpha_1,\alpha_2,q,\lambda)$
consisting of all $z\in\C$ for which 
\begin{equation}
\label{condition1}
\alpha_1 \log M(|z|,f)
\leq \left| \frac{zf'(z)}{f(z)}\right|
\leq \alpha_2 \log M(|z|,f) ,
\end{equation}
\begin{equation}
|f(z)|\geq |z|^q \label{condition2}
\end{equation}
and
\begin{equation}
\left|\frac{\zeta f'(\zeta)}{f(\zeta)}\right|
\leq \alpha_2 \log M(|\zeta|,f)
\quad\text{for} \quad
\left| \zeta -z\right|\leq \lambda \frac{|z|}{ \log M(|z|,f)}.
\label{condition3}
\end{equation}
Of course, the right inequality of~\eqref{condition1} is a special case
of~\eqref{condition3}.

For $R>0$ we put $A(R)=\left\{z\in\C: R\leq |z|\leq 2R\right\}$.
For measurable sets $X,Y\subset\C$ the density of $X$ in $Y$ is defined by
\[
\dens (X,Y)=\frac{\area(X\cap Y)}{\area(Y)}.
\]
\begin{theorem} \label{theorem1}
Let $f$ be an entire function
satisfying~\eqref{regular}. Then there exists $\alpha_1,\alpha_2,\eta>0$
such that if $q,\lambda\geq 0$, then $\dens(T(f,\alpha_1,\alpha_2,q,\lambda),A(R))>\eta$ for sufficiently large~$R$.
\end{theorem}
The proof of Theorem~\ref{theorem1}
is largely based on ideas of Miles and Rossi~\cite{Miles96};
cf. the remark at the end of section~\ref{lowerboundld}.

For an introduction to the dynamics of transcencental entire functions
we refer to~\cite{Bergweiler93}. Results on the Hausdorff dimension
of Julia sets of entire functions are surveyed in~\cite{Stallard08}.

\begin{ack}
We thank Phil Rippon and Gwyneth Stallard for drawing our attention
to Jian-Hua Zheng's paper~\cite{Zheng06}.
\end{ack}

\section{The regularity condition}
\subsection{Discussion of the regularity condition}\label{discreg}
We note that~\eqref{regular}  implies that if 
$r\geq r_0$ and if the integer $n$ is chosen such that
$C^nr_0\leq r<C^{n+1} r_0$, then
\[
\log M(r,f)\leq \log M(C^{n+1} r_0,f)\leq B^{n+1} \log M(r_0,f).
\]
Since $n\leq (\log(r/r_0))/(\log C)$ this implies that
\[
\log \log M(r,f) \leq n\log B +O(1)\leq  \frac{\log B}{\log C} \log r+O(1)
\]
as $r\to\infty$. Hence the order $\rho(f)$ of $f$
satisfies $\rho(f)\leq (\log B)/(\log C)<\infty$. Similarly, the 
lower order 
\[
\lambda(f)=\liminf_{r\to\infty}\frac{\log \log M(r,f)}{\log r},
\]
satisfies $\lambda(f)\geq (\log A)/(\log C)>0$. 

We recall that the upper logarithmic density $\upperlogdens E$ of a 
(measurable) subset $E$ of $[1,\infty)$ is defined by.
\[
\upperlogdens E=\limsup_{r\to\infty}
\frac{1}{\log r}\int_{E\cap[1,r]}\frac{dt}{t}.
\]
It is well-known ~\cite[Lemma~4]{Hayman65} that 
if $f$ is an entire function of finite order $\rho(f)$,
then the set $E$ where the right inequality of~\eqref{regular} does not hold
satisfies
\[
\upperlogdens E\leq 
\frac{\rho(f) \log C}{\log B}.
\]
We see that $E$ is a ``small'' set if $B$ is large, and
thus for functions $f$ of finite 
order~\eqref{regular} can be interpreted as a regularity condition
for the growth of~$f$.
\subsection{Consequences of the regularity condition}
It follows from~\eqref{regular} that
\begin{equation}\label{1a}
A^n\log M(r,f)\leq \log M(C^nr,f)\leq B^n\log M(r,f)
\end{equation}
for $n\in\N$. We may thus assume without loss of generality that
the constants $A,B,C$ are larger than any preassigned number.
Denote by $T(r,f)$ the Nevanlinna characteristic of~$f$.
Using the inequality~\cite{Goldberg08,Hayman64}
\[
T(r,f)\leq \log^+ M(r,f)\leq \frac{R+r}{R-r} T(R,f)
\]
we see that there exists constants $A_T,B_T,C_T>1$ such that 
\[
A_T T(r,f)\leq  T(C_Tr,f)\leq B_T T(r,f)
\]
for large~$r$. For $a\in\C$ we denote by $n(r,a)$ the number of
$a$-points of $f$ in the closed disk of radius $r$ around~$0$
and put
\[
N(r,a)=\int_0^r \frac{n(t,a)-n(0,a)}{t} dt + n(0,a)\log r. 
\]
Denote by $E_V(f)$ the set of Valiron deficiencies of $f$;
that is, the set of all $a$ for which
\[
\liminf_{r\to\infty} \frac{N(r,a)}{T(r,f)} <1.
\]
It is well-known~\cite[p.~116]{Goldberg08} that $\area E_V(f)=0$.
For $a\in\C\setminus E_V(f)$ we have 
\[
N(r,a)\sim T(r,f) 
\]
as $r\to\infty$.
Thus there exists constants $A_N,B_N,C_N>1$ such that
if $a\in\C\setminus E_V(f)$, then
\[
A_N N(r,a)\leq  N(C_Nr,a)\leq B_N N(r,a)
\]
for sufficiently large~$r$, say $r\geq r(a)$.

We note that if $M>1$, then
\[
n(r,a)=\frac{1}{\log M} \int_r^{Mr} \frac{n(r,a)}{t} dt
\leq \frac{1}{\log M} \int_r^{Mr} \frac{n(t,a)}{t} dt
\leq  \frac{1}{\log M} N(Mr,a) 
\]
and 
\[
n(Mr,a)\geq \frac{1}{\log M} \int_{r}^{Mr} \frac{n(t,a)}{t} dt
\geq \frac{1}{\log M}\left(N(Mr,a)-N(r,a)\right)
\geq \frac{A_N-1}{\log M} N(r,a)
\]
for large~$r$.
With $M=C_N$ we see that if $a\notin E_V(f)$, then
\[
n(C_N r,a)  \leq  \frac{1}{\log C_N }  N(C_N^2 r,a)  
\leq  \frac{B_N^3}{\log C_N }  N(C_N^{-1} r,a) 
\leq \frac{B_N^3}{A_N-1} n(r,a)
\]
for large~$r$.
We obtain
\[
n(C_N^n,a)  \leq \left(\frac{B_N^3}{A_N-1} \right)^n n(r,a)
\]
and choosing $n$ such that $C_N^n\geq 2$ we obtain
\begin{equation}\label{1b}
n(2r,a)\leq K n(r,a)
\end{equation}
with a constant $K$ for large~$r$.
We conclude that
\begin{equation}\label{1c}
\log M(r,f)\leq 
3 T(2 r,f)\leq 4 N(2 r,a) \leq \frac{4\log 2}{A_N-1} n(4r,a) \leq \frac{4 K^2\log 2}{A_N-1} n(r,a)
\end{equation}
for $a\notin E_V(f)$ and large~$r$.

\section{Proof of Theorem~\ref{theorem2}}
\subsection{An upper bound for the logarithmic derivative}
In this section we consider the set 
\[
U_\tau(f)
=\left\{z\in\C: 
\left| \frac{zf'(z)}{f(z)}\right|
\leq \tau \log M(|z|,f)  \right\}.
\]
We shall only need that the 
right inequality of~\eqref{regular}  is satisfied; that is,
\[
\log M(Cr,f)\leq B\log M(r,f)
\]
for large~$r$. 
As before we deduce that
\begin{equation}\label{1d}
\log M(2r,f)\leq L\log M(r,f)
\end{equation}
for a constant $L$ and large~$r$. 
Since $N(r,a)\leq T(r,f)+O(1)\leq \log M(r,f)+O(1)$
by Nevanlinna's first fundamental theorem,
this implies that
\begin{equation}\label{1e}
n(r,a)
\leq  \frac{1}{\log 2} N(2r,a) 
\leq \frac{1}{\log 2} \log M(2r,f)+O(1)
\leq \frac{L}{\log 2} \log M(r,f)+O(1)
\end{equation}
for all $a\in\C$, provided $r$ is sufficiently large.

\begin{lemma} \label{upperbound}
Let $f$ be an entire satisfying~\eqref{1d}.
Then for each $\varepsilon>0$ there exists $\tau>0$ such that 
$\dens(U_\tau(f),A(R))\geq 1-\varepsilon$ for all large~$R$.
\end{lemma}
To prove this result,
we shall need the following result due to Fuchs and 
Macintyre~\cite{Fuchs40}. Here and in the following 
we denote by $D(a,r)$ the open disk of radius $r$ around
a point~$a$.
\begin{lemma} \label{fuchs}
Let $z_1,z_2,\dots,z_m\in\C$ and let $H>0$.
Then there exists $l\in\{1,2,\dots,m\}$, 
$c_1,c_2,\dots,c_l\in\C$ and $r_1,r_2,\dots,r_l>0$ satisfying
\[
\sum_{k=1}^l r_k^2 \leq 4H^2
\]
such that 
\[
\sum_{k=1}^m\frac{1}{|z-z_k|}\leq  \frac{2m}{H} 
\quad\text{for}\ 
z\in\C, z\notin \bigcup_{k=1}^l D(c_k,r_k).
\]
\end{lemma}

\begin{proof}[Proof of Lemma~\ref{upperbound}]
For $s>|z|$ we have~\cite[p.~88]{Goldberg08}
\begin{equation}\label{1f}
\left| \frac{f'(z)}{f(z)}\right| 
\leq \frac{4s}{(s-|z|)^2} T(s,f)+ 
\sum_{|z_j|\leq s} \frac{2}{|z-z_j|},
\end{equation}
where $(z_j)$ is the sequence of zeros of~$f$.
(As in~\cite{Goldberg08} we have assumed here that $f(0)=1$,
but we may do so without loss of generality.)
Now we choose $s=4R$ so that
\[
\frac{4s}{(s-|z|)^2}\leq \frac{16R}{(4R-2R)^2}=\frac{4}{R}
\]
for $z\in A(R)$. Hence
\[
\frac{4s}{(s-|z|)^2} T(s,f) \leq \frac{4T(4R,f)}{R}
 \leq \frac{4\log M(4R,f)}{R}
 \leq \frac{4L^2 \log M(R,f)}{R}
 \leq  8L^2 \frac{\log M(|z|,f)}{|z|}
\]
for $z\in A(R)$.
To estimate the sum on the right hand side of~\eqref{1f} we use
Lemma~\ref{fuchs} with $H=\frac12 \sqrt{3\varepsilon} R$
and $m=n(s,0)$. With the notation of this lemma we have
\[
\area\left( \bigcup_{k=1}^l D(c_k,r_k)\right) 
=\pi \sum_{k=1}^l r_k^2 \leq 4\pi H^2 = 3\varepsilon \pi R^2=
\varepsilon \area A(R) 
\]
and  
if $z\notin\bigcup_{k=1}^l D(c_k,r_k)$,
then
\[
\sum_{|z_j|<s} \frac{2}{|z-z_j|}
\leq \frac{4m}{H}=  \frac{8}{ \sqrt{3\varepsilon}}
\frac{n(4R,0)}{R}.
\]
Using~\eqref{1e} we see that 
\[
\sum_{|z_j|<s} \frac{2}{|z-z_j|}
\leq  \frac{16}{ \sqrt{3\varepsilon}} 
\frac{L \log M(4R,f)+O(1)}{|z|}
\leq  \frac{17 L^3}{ \sqrt{3\varepsilon}(\log 2)^2} 
\frac{\log M(|z|,f)}{|z|},
\]
provided $R$ is sufficiently large.
The conclusion follows with
\[
\tau=
8L^2 +   \frac{17 L^3}{\sqrt{3\varepsilon}(\log 2)^2}.
\]
\end{proof}
The proof actually yields the following result.
\begin{lemma} \label{upperbound2}
Let $f$ be an entire satisfying~\eqref{1d}.
Then for each $\varepsilon>0$ there exists $\tau>0$ such that 
if $R$ is sufficiently large, then there exist 
$l\leq n(4R,0)$ and $c_1,\dots,c_l\in D(0,4R)$ and 
$r_1,\dots,r_l>0$ such that 
\[
A(R)\setminus U_\tau(f) \subset 
\bigcup_{k=1}^l D(c_k,r_k)
\quad\text{and}\quad 
\sum_{k=1}^l r_k^2 \leq \varepsilon R^2
\]
\end{lemma}

\subsection{A lower bound for the logarithmic derivative}\label{lowerboundld}
The results of this subsection are minor modifications of results
of Miles and Rossi~\cite{Miles96}.
The differences between their results and the results below
are explained at the end of this subsection.

Let $f$ be an entire function of finite order $\rho(f)$
and denote by $n(r)$ the number 
of zeros of $f$ in the disk of radius $r$ around~$0$.

\begin{lemma} \label{lemma1}
Suppose that there exists $r_0>0$ and $K>1$ such that
\begin{equation}\label{2a}
n(2r)\leq K n(r)
\quad\text{for}\ r\geq r_0 .
\end{equation}
For $\mu>0$ let $F_\mu$ be the set of all $r\geq r_0$ for which
\begin{equation}\label{2b}
n(t)\leq \left(\frac{t}{r}\right)^\mu  n(r) 
\quad\text{for}\ t\geq r
\end{equation}
while 
\begin{equation}\label{2c}
n(t)\geq \left(\frac{t}{r}\right)^\mu  n(r)
\quad\text{for}\  r_0\leq t\leq r.
\end{equation}
Then, given $\delta>0$,  there exists $\mu>0$ such that
\[
\meas\left(F_\mu \cap [R,2R]\right) \geq (1-\delta)R
\]
for all $R\geq 2r_0$.
\end{lemma}
\begin{proof}
It follows from~\eqref{2a} that if $t\geq 2r$ and if $m\in\N$ is 
chosen such that $2^mr\leq t<2^{m+1}r$, then
\[
n(t)\leq K^{m+1} n\left(2^{-m-1}t\right) \leq K^{m+1} n(r)
\leq K^{2m} n(r) 
=\exp \left(2m\log K\right) n(r) .
\]
We also have $m\leq (\log (t/r))/(\log 2)$ and thus
\[
n(t)\leq
\exp \left(2 \log\left(\frac{t}{r}\right) \frac{\log K}{\log 2}\right) n(r)
=\left(\frac{t}{r}\right)^\frac{2\log K}{\log 2} n(r).
\]
We see that if 
\[
\mu \geq \frac{2\log K}{\log 2},
\]
then
\begin{equation}\label{2d}
n(t)\leq \left(\frac{t}{r}\right)^\mu  n(r) 
\quad\text{for}\ t\geq 2r
\end{equation}
so that condition~\eqref{2b} is satisfied as soon as
\begin{equation}\label{2d1}
n(t)\leq \left(\frac{t}{r}\right)^\mu  n(r) 
\quad\text{for}\ r\leq t\leq 2r.
\end{equation}
Let now $R\geq 2r_0$ and let $E_1$ be the set of all 
$r\in [R,2R]$ where~\eqref{2b} does not hold and let 
$E_2$ be the set of all
$r\in [R,2R]$ where~\eqref{2c} does not hold.
We shall show that 
\[
\meas E_1 \leq \tfrac12\delta R
\quad\text{and}\quad
\meas E_2 \leq \tfrac12\delta R
\]
if $\mu$ is chosen large enough.
The conclusion then follows.

To prove the claim about $E_1$ we may assume that $E_1\neq\emptyset$
and choose
\[
s_1\in E_1\cap \left[ \inf E_1,\inf E_1 +\tfrac18 \delta R\right] 
\]
Then there exists $t_1>s_1$ with 
\[
n(t_1)> \left(\frac{t_1}{s_1}\right)^\mu  n(s_1) 
\]
and because of~\eqref{2d1} we have $t_1\leq 2s_1\leq 4R$.
Inductively we define 
\[
s_k\in E_1\cap \left[ \inf(E_1\cap[t_{k-1},2R]),
\inf(E_1\cap[t_{k-1},2R]) +2^{-k-2} \delta R\right]
\]
and choose $t_k\in (s_k,2s_k]$ with
\begin{equation}\label{2e}
n(t_k)> \left(\frac{t_k}{s_k}\right)^\mu  n(s_k),
\end{equation}
as long as $E_1\cap[t_{k-1},2R]\neq \emptyset$. However,
noting that $n(t_k)>n(s_k)\geq n(t_{k-1})$ and 
$n(t_k)\leq n(2s_k)\leq n(4R)$ we see that the process terminates so
that there exists $N\in\N$ with 
$E_1\cap[t_N,2R]=\emptyset$ and
\[
E_1\subset \bigcup_{k=1}^N \left[ s_k-2^{-k-2}\delta R,t_k\right].
\]
Since $s_k\geq t_{k-1}$ it follows from~\eqref{2e} that
\[
n(t_k)> \left(\frac{t_k}{s_k}\right)^\mu  n(t_{k-1})
\]
and hence that
\[
n(4R)\geq n(t_N)> n(R) \prod_{k=1}^N \left(\frac{t_k}{s_k}\right)^\mu .
\]
Since $n(4R)\leq K^2 n(r)$ by~\eqref{2a}  this yields
\[
\prod_{k=1}^N \left(\frac{t_k}{s_k}\right)^\mu \leq K^2
\]
and thus
\[
\mu \sum_{k=1}^N \log \frac{t_k}{s_k} \leq 2\log K.
\]
Since $0\leq t_k-s_k\leq s_k\leq 2R$ and since
$\log(1+x)\geq x\log 2$ for $0\leq x\leq 1$ we have
\[
\log \frac{t_k}{s_k}
=
\log  \left(1+ \frac{t_k-s_k}{s_k}\right)
\geq \frac{t_k-s_k}{s_k}\log 2
\geq (t_k-s_k)\frac{\log 2}{2R}
\]
and thus 
\[
\sum_{k=1}^N (t_k-s_k) \leq \frac{2R}{\log 2}
\sum_{k=1}^N \log \frac{t_k}{s_k}\leq 
\frac{4R\log K}{\mu \log 2}.
\]
Choosing $\mu >(16 \log K)/(\delta\log 2)$ we obtain
\[
\sum_{k=1}^N (t_k-s_k) \leq \tfrac14 \delta R
\]
and thus  
\[
\meas E_1 \leq \sum_{k=1}^N (t_k-s_k+2^{-k-2}\delta R)\leq \tfrac12\delta R.
\]

The estimate for $E_2$ is similar.
Here we choose
\[
s_1\in E_2\cap \left[ \sup E_2-\tfrac18 \delta R,\sup E_2 \right] 
\]
and  $r_1\in [r_0,s_1)$ with 
\[
n(r_1)< \left(\frac{r_1}{s_1}\right)^\mu  n(s_1) .
\]
It follows from~\eqref{2d} that $r_1\in \left[\frac12 s_1,s_1\right)$.
Inductively we choose 
\[
s_k\in E_2\cap \left[ \sup (E_2\cap [R,r_{k-1}])-2^{-k-2} \delta R,\sup (E_2\cap [R,r_{k-1}]) \right] 
\]
and  $r_k\in \left[\frac12 s_k,s_k\right)$ with 
\[
n(r_k)< \left(\frac{r_k}{s_k}\right)^\mu  n(s_k) .
\]
Again the process stops and there exists $N\in\N$ with 
\[
E_2\subset \bigcup_{k=1}^N \left[ r_k,s_k+2^{-k-2}\delta R\right]
\]
and
\[
n\left(\tfrac12 R\right) \leq n(r_N) \leq 
n(s_1) 
\prod_{k=1}^N \left(\frac{r_k}{s_k}\right)^\mu 
\leq 
K^2 n\left(\tfrac12 R\right)
\prod_{k=1}^N \left(\frac{r_k}{s_k}\right)^\mu.
\]
Thus
\[
\mu \sum_{k=1}^N \log \frac{s_k}{r_k} \leq 2\log K
\]
and as before this yields
\[
\meas E_2 \leq \sum_{k=1}^N (s_k-r_k+2^{-k-2}\delta R)\leq \tfrac12\delta.
\]
\end{proof}

We note that follows from~\eqref{2b} and~\eqref{2c} that $n(t)$ 
is continuous at $t=r$, meaning that there is no zero of $f$ on the circle of radius $r$
around~$0$. We also note that~\eqref{2b} and~\eqref{2c} remain valid if $\mu$ is 
replaced by a larger number.

We shall assume that $f(0)\neq 0$ and denote by $(z_j)$ the sequence of zeros 
of~$f$, ordered such that $|z_1|\leq|z_2|\leq \dots$. Replacing $K$ by a larger 
number if necessary, we may assume that~\eqref{2a} holds with $r_0=|z_1|$. 

For $0<\beta <1$ and $r\geq 0$ we put
\[
U(r)=\left\{\theta\in [0,2\pi]: \left|\frac{f'(re^{i\theta})}{f(re^{i\theta})}\right|\geq 
\beta \frac{n(r)}{r} \right\}.
\]
\begin{lemma} \label{lemma2}
If~\eqref{2b} and~\eqref{2c} hold for some $\mu\geq \max\{\rho(f),1\}$, then
\begin{equation}\label{2k}
\meas U(r) \geq \frac{2\pi(1-\beta)^2}{(\beta +3\pi\mu)^2},
\end{equation}
provided $r$ is sufficiently large.
\end{lemma}
\begin{proof}
We put $q=[\mu+1]$ and write
\[
f(z)=e^{P(z)}\prod_{j=1}^\infty E\left(\frac{z}{z_j},q\right),
\]
where $P(z)=\sum_{m=0}^q a_mz^m$
is a polynomial of degree at most $q$ and where
$E(\cdot,q)$ denotes the Weierstra{\ss} primary factor.
Note that $q$ will in general be much larger than $\rho(f)$ 
so that the above form of $f$ is not the usual Hadamard 
factorization. We put
\[
L(\theta)= re^{i\theta}\frac{f'(re^{i\theta})}{f(re^{i\theta})}.
\]
Then
\[
L(\theta)
=re^{i\theta}P'(re^{i\theta})+\sum_{m=-\infty}^\infty b_m(r) e^{im\theta}
\]
where~\cite[p.~350]{Townsend87}
\[
b_m(r)=-\sum_{|z_j|>r} \left(\frac{z_j}{r}\right)^{-m}
\quad\text{for}\ m>q
\]
while
\[
b_m(r)=\sum_{|z_j|< r} \left(\frac{z_j}{r}\right)^{-m}
\quad\text{for}\ m<0.
\]
For $m>q$ we deduce from~\eqref{2b} that
\[
\begin{aligned}
|b_m(r)| 
&\leq \sum_{|z_j|>r} \left(\frac{|z_j|}{r}\right)^{-m}\\
&= \int_r^\infty \left(\frac{t}{r}\right)^{-m} dn(t)\\
&= -n(r)+m \int_r^\infty \left(\frac{t}{r}\right)^{-m} n(t)\frac{dt}{t}\\
&\leq -n(r)+m \, n(r) \int_r^\infty \left(\frac{t}{r}\right)^{\mu-m} \frac{dt}{t}\\
&= n(r)\frac{\mu}{m-\mu}.
\end{aligned}
\]
Thus
\[
\sum_{m>q} |b_m(r)|^2 \leq \frac{\pi^2}{6}\mu^2 n(r)^2.
\]
Similarly we find for $m<0$ that 
\[
\begin{aligned}
|b_m(r)| 
&\leq \int_{r_0/2}^r \left(\frac{t}{r}\right)^{-m} dn(t)\\
&= n(r)+m \int_{r_0}^r \left(\frac{t}{r}\right)^{-m} n(t)\frac{dt}{t}\\
&\leq n(r)+ m \, n(r) \int_{r_0}^r \left(\frac{t}{r}\right)^{\mu-m} \frac{dt}{t}\\
&= n(r)\left( \frac{\mu}{\mu-m} +m \left(\frac{r_0}{r}\right)^{\mu-m}\right).
\end{aligned}
\]
For large $r$ we thus have 
\[
|b_m(r)| \leq \frac{2\mu}{\mu-m} n(r)
\]
for all $m<0$ and this yields
\[
\sum_{m<0} |b_m(r)|^2 \leq \frac{2\pi^2}{3}\mu^2 n(r)^2.
\]
With
\[
g(\theta)=re^{i\theta} P'(re^{i\theta})+\sum_{m=0}^q b_m(r) e^{im\theta}
=\sum_{m=0}^q \left(ma_m r^m +b_m(r)\right) e^{im\theta}
\]
we thus have 
\begin{equation}\label{2l1}
L(\theta)=g(\theta)+s(\theta)
\end{equation}
where
\[
\|s\|_2^2 
=\frac{1}{2\pi}\int_0^{2\pi}|s(\theta)|^2d\theta
=\sum_{m<0}|b_m(r)|^2 + \sum_{m>q}|b_m(r)|^2
\leq \pi^2\mu^2 n(r)^2
\]
so that
\begin{equation}\label{2m1}
\|s\|_2\leq \pi\mu \,n(r).
\end{equation}

In order to estimate $\|g\|_2$ we write
\[
|g(\theta)|^2= \sum_{m=-q}^q h_m(r) e^{im\theta}
\]
and note that $h_{-m}(r)=\overline{h_m(r)}$ and
\begin{equation}\label{2n}
| h_m(r)| 
=\left|\frac{1}{2\pi}\int_0^{2\pi}|g(\theta)|^2 e^{-im\theta} d\theta\right|
\leq \|g\|_2^2
=\frac{1}{2\pi}\int_0^{2\pi}|g(\theta)|^2d\theta =h_0(r)
\end{equation}
for all~$m$.
Let now $V(r)=[0,2\pi]\setminus U(r)$ so that
\begin{equation}\label{2m2}
|L(\theta)|< \beta n(r)
\quad\text{for}\ \theta\in V(r).
\end{equation}
Since 
\[
0=\int_0^{2\pi} h_m(r)e^{im\theta} d\theta
=\int_{U(r)}  h_m(r)e^{im\theta}d\theta + 
\int_{V(r)}  h_m(r)e^{im\theta} d\theta
\]
for $m\neq 0$ we deduce from~\eqref{2n} that
\begin{equation}\label{2o}
\begin{aligned}
\int_{V(r)}|g(\theta)|^2 d\theta 
&= \int_{V(r)} h_0(r) d\theta -
\sum_{1\leq |m|\leq q}
\int_{U(r)}  h_m(r)e^{im\theta}d\theta\\
&\geq h_0(r) \meas V(r) -
\sum_{1\leq |m|\leq q}
|h_m(r)|  \meas U(r) \\
&\geq  h_0(r) \meas V(r) -2q   h_0(r)\meas U(r) \\
&= h_0(r)\left(  \meas V(r)-2q \meas U(r)\right) .
\end{aligned}
\end{equation}
If $\meas U(r) \geq \pi/(2q+1)$, then~\eqref{2k} follows
since $q\leq \mu+1$ so that $2q+1\leq 2\mu+3\leq 5\mu$ and this
yields
\[
\frac{\pi}{2q+1} \geq
\frac{2\pi(1-\beta)^2}{(\beta +3\pi\mu)^2}.
\]
We may thus assume that $\meas U(r) <  \pi/(2q+1)$ so that 
$\meas V(r) > 2\pi - \pi/(2q+1)$.
We deduce from~\eqref{2n} and~\eqref{2o} that
\[
\int_{V(r)}|g(\theta)|^2 d\theta 
\geq  \left( 2\pi - \frac{\pi}{2q+1} -2q  \frac{\pi}{2q+1} \right) 
h_0(r)=\pi \|g\|_2^2.
\]
Using~\eqref{2l1}, \eqref{2m1} and~\eqref{2m2} we find that
\[
\begin{aligned}
\frac{1}{\sqrt{2}}  \|g\|_2
&\leq
\left(\frac{1}{2\pi} \int_{V(r)}|g(\theta)|^2 d\theta \right)^{1/2}\\
&\leq\left(\frac{1}{2\pi} \int_{V(r)} |L(\theta)|^2
d\theta \right)^{1/2} 
+ \left(\frac{1}{2\pi} \int_{V(r)} |s(\theta)|^2d\theta \right)^{1/2}\\
&\leq 
\beta n(r) + \pi\mu n(r)
\end{aligned}
\]
and hence 
\[
\|g\|_2 \leq \left(\beta +2\pi\mu\right) n(r).
\]
Combining this with~\eqref{2l1} and~\eqref{2m1} we conclude that
\begin{equation}\label{2r}
\|L\|_2 \leq  \|g\|_2 + \|s\|_2
\leq \left(\beta +3\pi\mu\right) n(r).
\end{equation}
On the other hand, it follows from the argument principle
that 
\[
n(r) 
= \frac{1}{2\pi i} \int_{|z|=r} \frac{f'(z)}{f(z)} dz
= \frac{1}{2\pi} \int_0^{2\pi} L(\theta) d\theta
= \frac{1}{2\pi} \int_{U(r)} L(\theta) d\theta 
+  \frac{1}{2\pi} \int_{V(r)} L(\theta) d\theta.
\]
Now the Cauchy-Schwarz inequality,
\eqref{2m2} and~\eqref{2r} yield 
\[
\begin{aligned}
n(r) 
& \leq
\frac{1}{2\pi}
\left( \int_{U(r)}   d\theta\right)^{1/2}
\left( \int_{U(r)}  |L(\theta)|^2  d\theta\right)^{1/2} +\beta n(r) \\
& \leq 
\frac{1}{\sqrt{2\pi}} \sqrt{\meas U(r)}
\|L\|_2   +\beta n(r) \\
& \leq
\frac{1}{\sqrt{2\pi}} \sqrt{\meas U(r)}
\left(\beta +3\pi\mu\right) n(r)  +\beta n(r).
\end{aligned}
\]
Hence 
\[
\meas U(r)\geq \frac{2\pi (1-\beta)^2}{(\beta +3\pi\mu)^2}.
\]
\end{proof}
\begin{remark}
It was shown at the beginning of the proof of 
Lemma~\ref{lemma1} that if $n(r)$ satisfies~\eqref{2a},
then there exists $\rho>0$ such that
\begin{equation}\label{2x}
n(r)= O\left( r^\tau\right)
\end{equation}
as $r\to\infty$. (In fact, the argument shows that we can
take $\tau=(2\log K)/(\log 2)$, and a slightly more careful
estimate will give $\tau=(\log K)/(\log 2)$.)

Miles and Rossi~\cite{Miles96} show that if $n(r)$ satisfies~\eqref{2x},
then~\eqref{2b} and~\eqref{2c} hold on a set 
of logarithmic density $1-\delta$ if $\mu$ is sufficiently
large. They then use this to show that~\eqref{2k} holds on a set
of logarithmic density $1-\delta$.

For our applications, however, a set of positive 
logarithmic density is not sufficient. Therefore we introduced
the additional hypothesis~\eqref{2a}. Lemma~\ref{lemma1}
says that with this additional hypothesis~\eqref{2k} holds 
on a set of density $1-\delta$.

The proof of Lemma~\ref{lemma2}, which says 
that~\eqref{2b} and~\eqref{2c} imply~\eqref{2k}, is 
essentially the same 
as that of Miles and Rossi~\cite{Miles96} and it is included here
only for completeness.

We also note that~\eqref{2x} implies that there exists a constant $K$ such that~\eqref{2a} 
holds on a set of positive density. In fact, by taking $K$ large this density can
be taken arbitrarily close to~$1$.
\end{remark}

\subsection{Completion of the proof of Theorem~\ref{theorem1}}
Let $a\in \C\setminus E_V(f)$. 
By~\eqref{1b} we can apply Lemma~\ref{lemma1} to 
$f-a$.
With
\[
\gamma= \frac{\beta (A_N-1)}{4K^2\log 2}
\quad\text{and} \quad
c=\frac{1-\delta}{2} \frac{(1-\beta)^2}{(\beta +3\pi\mu)^2}
\]
and with
\[
V_\gamma(a) =\left\{z\in\C: 
\left| \frac{z f'(z)}{f(z)-a}\right|
\geq \gamma \log M(|z|,f)  \right\}
\]
we deduce from  Lemmas~\ref{lemma1} and~\ref{lemma2} and from~\eqref{1c}
that
$\dens(V_\gamma(a),A(R))\geq c$ for large~$R$.
We apply Lemma~\ref{upperbound} with $\varepsilon=\frac14 c$ 
to $f-a$ and with 
\[
U_\tau(f-a) =\left\{z\in\C: 
\left| \frac{zf'(z)}{f(z)-a}\right|
\leq \tau  \log M(|z|,f)  \right\}
\]
we obtain 
$\dens(U_\tau(f-a),A(R))\geq 1-\varepsilon$
if $\tau$ is sufficiently large.

We put $d=\tau/\gamma$, fix $m\geq 1/\varepsilon$
and choose $a_1,\dots,a_m\in D(0,2dm)\setminus E_V(f)$
with $|a_j-a_k|\geq 2d$ for $j\neq k$.
For $1\leq j\leq m$ we put 
\[
C_j=\left\{z\in\C: 
|f(z)-a_j|\leq d\right\}.
\]
Then the $C_j$ are pairwise disjoint and thus there exists
$j=j(R)$ with $\dens(C_j, A(R) )\leq \varepsilon$.
With 
\[
W(a_j) = \left( U_\tau(f-a_j)\cap  U_\tau(f))\cap V_\gamma(a_j) \right) \setminus  C_j
\]
we deduce from Lemma~\ref{upperbound}  that $\dens(W(a_j),A(R))\geq \frac14 c$.
For $z\in W(a_j)\cap A(R)$ we have
\[
\gamma  \log M(|z|,f)  
\leq 
\left| \frac{zf'(z)}{f(z)-a_j}\right|
\leq  \frac{|zf'(z)|}{d}
\]
and thus
\[
|zf'(z)|
\geq
d \gamma   \log M(|z|,f) 
=\tau  \log M(|z|,f)  
\geq 
\left| \frac{zf'(z)}{f(z)}\right|.
\]
Hence $|f(z)|\geq 1$ for $z\in W(a_j)\cap A(R)$. 
Moreover, if $|f(z)|\leq 4dm$, then we have
\[
\left| \frac{zf'(z)}{f(z)}\right|=
\left| \frac{zf'(z)}{f(z)-a_j}\frac{f(z)-a_j}{f(z)}\right|
\geq \frac{1}{4m} \left| \frac{zf'(z)}{f(z)-a_j}\right|
\geq \frac{\gamma}{4m}   \log M(|z|,f) 
\]
and if $|f(z)|\geq 4dm$, 
then  $|f(z)|\geq 2|a_j|$ and thus
\[
\left| \frac{zf'(z)}{f(z)}\right|=
\left| \frac{zf'(z)}{f(z)-a_j}\left(1- \frac{a_j}{f(z)}\right)\right|
\geq \frac12 \left| \frac{zf'(z)}{f(z)-a_j}\right|
\geq \frac{\gamma}{2}  \log M(|z|,f) .
\]
Summarizing the above estimates we obtain with $\sigma= \gamma/(4m)$ 
that if $z\in W(a_j)\cap A(R)$, then
\begin{equation}\label{2y}
\sigma  \log M(|z|,f) 
\leq
\left| \frac{zf'(z)}{f(z)}\right|
\leq \tau  \log M(|z|,f)  
\quad\text{and}\quad
|f(z)|\geq 1.
\end{equation}
Noting that $|z|^0=1$ we have thus proved that if $f$ satisfies~\eqref{1a}, then 
\[
\dens(T(f,\sigma,\tau,0,0),A(R))\geq\frac14 c
\]
for all large~$R$. In other words, we have proved the special case $q=\lambda=0$ of our
theorem.
We may apply this result to 
\[
g(z)=\frac{f(z)-a}{P(z)}
\]
where $a$ is chosen such that $f$ has infinitely many $a$-points and
where $P$ is a polynomial of degree greater than
$q$ whose zeros are $a$-points of~$f$.
In fact, we have 
\[
\log M(r,g)=\log M(r,f)+O(\log r) =(1+o(1))\log M(r,f)
\] 
as $r\to\infty$ so that~\eqref{1a} holds with $f$ replaced by $g$ if the constants
$A$ and $C$ are slightly adjusted.
Hence $\dens(T(g,\sigma^*,\tau^*,0,0),A(R))\geq\eta$ if $0<\sigma^*<\sigma$, $\tau^*>\tau$ 
and $0<\eta<\frac14 c$,
provided $R$ is large enough.

Now 
\[
\frac{zg'(z)}{g(z)}
=\frac{zf'(z)}{f(z)-a}-\frac{zP'(z)}{P(z)}
=\frac{zf'(z)}{f(z)}\frac{f(z)}{f(z)-a}-\frac{zP'(z)}{P(z)}.
\]
For $z\in T(g,\sigma^*,\tau^*,0,0)$ we have $|g(z)|\geq 1$ and 
thus $|f(z)|=|P(z)g(z)+a|\geq |z|^q$, provided $|z|$ is 
sufficiently large.
Thus 
\[
\frac{f(z)}{f(z)-a}\to 1
\]
as $|z|\to \infty$, $z\in T(g,\sigma^*,\tau^*,0,0)$. Since 
\[
\frac{zP'(z)}{P(z)}\to \deg{P}
\]
as $|z|\to \infty$ we conclude that if $0<\alpha_1<\sigma^*$ and $\alpha_2>\tau^*$, then
\[
T(g,\sigma^*,\tau^*,0,0)\setminus D(0,S)\subset T(f,\alpha_1,\alpha_2,q,0)
\]
for large $S$ and hence 
$\dens(T(f,\alpha_1,\alpha_2,q,0),A(R))\geq \eta$
for large~$R$. This is the special case $\lambda=0$ of our theorem.

In order to obtain this result for general $\lambda$, we apply Lemma~\ref{upperbound2}.
We note that if $z\in A(R)$, then $|z|/\log M(|z|,f)\leq 2R/\log M(R,f)$.
This implies that if $c_1,\dots,c_l$ and $r_1,\dots,r_l$ are as in Lemma~\ref{upperbound2}
and if $z\in A(R)\setminus \bigcup_{k=1}^l D\left(c_k,r_k+2\lambda R/\log M(R,f)\right)$, 
then $D(z,\lambda |z|/\log M(|z|,f))\cap \bigcup_{k=1}^l D\left(c_k,r_k\right)=\emptyset$.
Thus
\[
\left|\frac{\zeta f'(\zeta)}{f(\zeta)}\right|\leq \tau \frac{\log M(|\zeta|,f)}{|\zeta|}
\quad\text{for} \quad 
z\in D\left(z,\lambda\frac{|z|}{\log M(|z|,f)}\right),
\]
provided
\[
z\in A(R)\setminus\bigcup_{k=1}^l D\left(c_k,r_k+\frac{2\lambda R}{\log M(R,f)}\right).
\]
Using $(a+b)^2\leq 2(a^2+b^2)$ and~\eqref{1d} and~\eqref{1e} we see that
\[
\begin{aligned}
\sum_{k=1}^l \left(r_k+ \frac{2\lambda R}{\log M(R,f)}\right)^2
&\leq 
2 \sum_{k=1}^l r_k^2 + \frac{8\lambda^2 R^2}{(\log M(R,f))^2} n(4R,0)\\
&
\leq 
2\varepsilon R^2 + \frac{8\lambda^2 L^3 R^2}{(\log 2)\log M(R,f)}\\
&\leq 3\varepsilon R^2.
\end{aligned}
\]
We see that the density of the 
set of all $z\in A(R)$ for which~\eqref{condition3} fails can be made arbitrarily small
by choosing $\alpha_2$ large. The conclusion follows.

\section{Auxiliary results for the proof of Theorem~\ref{theorem2}}\label{lemmas}
The following lemma can be proved by a simple compactness argument.
\begin{lemma}\label{koebe1}
Let $\Omega$ be a domain let $Q$ be a compact subset of $\Omega$.
Then there exists a positive constant $C$ such that if $f$ is univalent in $\Omega$ and
$z,\zeta\in Q$, then $|f'(\zeta)|\leq C|f'(z)|$.
\end{lemma}
In principle this lemma would be sufficient for our purposes, but we note that 
the classical Koebe distortion theorem gives explicit estimates in the case where 
$\Omega$ is a disk.
\begin{lemma}\label{koebe2}
Let $f$ be univalent in  $D(a,r)$ and let  $z\in \overline{D(a,\rho r)}$,
where
$0<\rho<1$. 
Then
\[
\frac{1-\rho}{(1+\rho)^3}
|f'(a)|
\leq |f'(z)|
\leq
\frac{1+\rho}{(1-\rho)^3}
|f'(a)|.
\]
\end{lemma}
We shall use the following version of the Ahlfors islands theorem;
cf.~\cite[Theorem~6.2]{Hayman64}.
\begin{lemma}\label{ahlfors}
Let $D_1,D_2,D_3$ be Jordan domains with pairwise disjoint closures.
Then there exists $\mu>0$ with the following property: if $a\in\C$, $r>0$ and
$f:D(a,r)\to\C$ is a holomorphic function satisfying 
\[
\frac{|f'(a)|}{1+|f(a)|^2}\geq \frac{\mu}{r},
\]
then $D(a,r)$ has a subdomain which is mapped bijectively 
onto one of the domains~$D_\nu$.
\end{lemma}
To estimate the Hausdorff dimension we will use a
result of McMullen~\cite{McMullen87}.
In order to state it,
consider for $l\in \N$ a
collection ${\mathcal E}_l$ of disjoint compact subsets of $\R^n$
such that the following two conditions are satisfied:
\begin{enumerate}
\item[(a)] every element of ${\mathcal E}_{l+1}$ is
contained in a
unique element of ${\mathcal E}_l$;
\item[(b)] every element of ${\mathcal E}_l$ contains at least one element of
${\mathcal E}_{l+1}$.\end{enumerate}
Denote by $E_l$ the union of all elements of ${\mathcal E}_l$ and
put $E=\bigcap^\infty_{l=1} E_l$. Suppose that
$(\Delta_l)$ and
$(d_l)$ are sequences of positive real numbers such that if $F\in
{\mathcal E}_l$, then
\[\dens(E_{l+1},F)\geq\Delta_l\]
and
\[\diam F \leq d_l.\]
Then we have the following result~\cite[Proposition~2.2]{McMullen87}.
\begin{lemma} \label{lemmamcm}
Let~$E$, ${\mathcal E}_l$, $\Delta_l$ and $d_l$ be as above. Then
\[\limsup_{l\to \infty}\frac{\sum^{l}_{j=1}\;|\log\Delta
_j|}{|\log d_l|}\geq n-\dim E.\]
\end{lemma}
The following result is due to Zheng~\cite[Corollary~5]{Zheng06}.
\begin{lemma} \label{lemmars}
Let $f$ be an entire function 
satisfying~\eqref{regular}.
Then the Fatou set of $f$ has no multiply connected components.
\end{lemma}
Actually Zheng requires only that the left inequality of~\eqref{regular}
holds. More precisely, he assumes that there exists $d>1$ such that
$\log M(2r,f)\geq d \log M(r,f)$ for all large~$r$,
but we may replace $M(2r,f)$ by $M(Cr,f)$ here if $C>1$.

\section{Proof of Theorem~\ref{theorem2}}\label{proof2}
Let $f$ be an entire function satisfying~\eqref{regular}
and let $\alpha_1,\alpha_2,\eta$ be as in Theorem~\ref{theorem1}.
We apply Lemma~\ref{ahlfors} to the domains 
\[
D_\nu=\{z
\in\C
: |\re z|< 1, |\im z -8\pi \nu|< 3\pi\},
\quad\text{for}\ \nu\in\{1,2,3\}.
\]
Let $\mu$ be such that the conclusion of Lemma~\ref{ahlfors} is
satisfied for these domains and let
$\lambda=2L\mu/\alpha_1$ where $L$ is the constant from~\eqref{1d}.
We will apply Theorem~\ref{theorem1} with this value of $\lambda$
and with $q=8$.

For large $R$ we put $t=t(R)=2 \mu R/(\alpha_1\log M(R,f))$.
Note that~\eqref{condition3} implies in particular that 
$f$ has no zeros in $D(z,\lambda|z|/\log M(|z|,f))$
if $z\in T(f,\alpha_1,\alpha_2,q,\lambda)$.
Since 
\[
\frac{\lambda |z|}{\log M(|z|,f)}
\geq
\frac{\lambda R}{\log M(2R,f)}
\geq
\frac{\lambda R}{L\log M(R,f)}
=
t(R)
\]
for $z\in A(R)$ by~\eqref{1d} this implies that a branch of
$\log f$ can be defined 
in the disk $D(z,t(R))$, provided 
$z\in T(f,\alpha_1,\alpha_2,q,\lambda)\cap A(R)$.
\begin{lemma}\label{applyahlfors}
Let $a\in T(f,\alpha_1,\alpha_2,q,\lambda)\cap A(R)$. 
If $R$ is sufficiently large, then 
$D(a,t(R))$ contains a subdomain $U$ 
such that 
$\log f$  maps 
$U$ bijectively onto
one of the domains
\[
\begin{aligned}
\Omega_\nu(a)
&=\log f(a)+D_\nu\\
&=\{z
\in\C
: |\re (z-\log f(a))|<1,|\im (z-\log f(a))-8\pi \nu|<3\pi\}.
\end{aligned}
\]
Moreover, there exist 
$\beta,\gamma>0$
such that 
if $V$ is the subset of $U$ which is mapped onto 
\[
Q_\nu(a)=\{z
\in\C
: 
0\leq \re (z-\log f(a))\leq \log 2, |\im (z-\log f(a))-8\pi \nu|\leq 2\pi\},
\]
then 
$\area V\geq \beta\, t(R)^2$ and
\begin{equation} \label{gamma}
\left|\frac{f'(z)}{f(z)}\right|\geq \frac{\gamma}{t(R)}
\quad\text{for}\ z\in V.
\end{equation}
\end{lemma}
\begin{proof}
Let $h:D(a,t)\to\C$, $h(z)=\log f(z)-\log f(a)$.
Then $h(a)=0$ and thus
\[
\frac{|h'(a)|}{1+|h(a)|^2}=|h'(a)|=\frac{|f'(a)|}{|f(a)|}
\geq  \alpha_1\frac{\log M(|z|,f)}{|z|}
\geq \frac{\alpha_1}{2}\frac{\log M(R,f)}{R}
=\frac{\mu}{t}
\]
Lemma~\ref{ahlfors} implies that there exists a subdomain $U$ of 
$D(a,t)$ which is mapped by $h$ bijectively onto one of the three domain $D_\nu$ 
occuring in this lemma. It follows that $\log f$ maps $U$ 
bijectively onto one of the domains $\log f(a)+D_\nu$.

We have 
\[
\begin{aligned}
4\pi\log 2
&
=\area Q_\nu 
\\ &
=\area h(V)
\\ &
\leq\sup_{z\in V}|h'(z)|^2 \area V
\\ &
\leq\sup_{z\in V}\left(\alpha_2 \frac{\log M(|z|,f)}{|z|}\right)^2 \area V
\\ &
\leq\left(\alpha_2 L\frac{\log M(R,f)}{R}\right)^2 \area V
\\ &
=\left(\frac{2\mu\alpha_2 L}{\alpha_1 t}\right)^2 \area V
\end{aligned}
\]
With $\beta=(\pi\alpha_1^2 \log 2)/(\mu\alpha_2 L)^2$ we thus have
$\area V\geq \beta\, t^2$.

By Lemma~\ref{koebe1}, there exists a constant $C>1$ such if $\phi$ denotes 
the branch of the inverse of $h$ which maps $D_\nu$ to~$U$, then
$|\phi'(\zeta)|\leq C|\phi'(z)|$ for $z,\zeta\in Q_\nu$. 
It follows 
that $|h'(\zeta)|\leq C|h'(z)|$ for $z,\zeta\in V$. 
Thus
\[
4\pi\log 2\leq \sup_{z\in V}|h'(z)|^2 \area V \leq C^2 \inf_{z\in V}|h'(z)|^2 \area V
\leq C^2\pi t^2 \inf_{z\in V}|h'(z)|^2 
\]
so that
\[
\inf_{z\in V}|h'(z)|\geq \frac{2 \sqrt{\log 2}}{Ct}.
\]
Since $h'=f'/f$ we see that~\eqref{gamma}
follows with $\gamma=2 \sqrt{\log 2}/C$.
\end{proof}
If $U$ is as in
Lemma~\ref{applyahlfors}, then $f$ maps 
$U$ onto the annulus 
\[
A'=\{z
\in\C
:e^{-1} |f(a)|  <|z|<e |f(a)| \}
\]
and $V$ onto 
$A(|f(a)|)$. 
If $D(b,r)$ is a disk contained in $A(|f(a)|)$, then 
there is a branch of the logarithm mapping $D(b,r)$ into~$Q_\nu(a)$
and thus a branch $\theta$ of the inverse of $f$ mapping 
$D(b,r)$ into~$V$.
Since $D(b,2r)\subset A'$, the branch of the logarithm extends to a map from
$D(b,2r)$ into $\Omega_\nu(a)$ and $\theta$ extends to a map from
$D(b,2r)$ into~$U$.

We now show that there are comparatively many disks disjoint $D(a,t)$ 
to which Lemma~\ref{applyahlfors}  can be applied.
\begin{lemma}\label{disks}
Let $\eta$ be as in Theorem~\ref{theorem1}.
For sufficiently large $R$ there exist a positive integer
 $m(R)$ satisfying
\[
m(R)\geq \frac{\eta}{2}\left(\frac{R}{t(R)}\right)^2
\]
such that there are
$m(R)$ points 
\[
a_j=a_j(R)
\in T(f,\alpha_1,\alpha_2,q,\lambda)\cap A(R),
\quad j=1,\dots,m(R),
\]
satisfying
$D(a_j,t(R))\subset A(R)$ for all $j$ and
$D(a_j,t(R))\cap D(a_k,t(R))=\emptyset$ for $j\neq k$. 
\end{lemma}
\begin{proof}
Let $m$ be the maximum number of points 
$a_1,\dots,a_m\in T(f,\alpha_1,\alpha_2,q,\lambda)\cap A(R)$
which satisfy the conclusion of the lemma. Then
\[
T(f,\alpha_1,\alpha_2,q,\lambda)\cap A(R) \subset 
\bigcup_{k=1}^m D(a_k,2 t(R))\cup \left\{z\in A(R):\dist(z,\partial A(R))
\leq t(R)\right\},
\]
since a point contained in the left but not in the right side could be 
added to the collection $a_1,\dots,a_m$.
It follows that
\[
\area 
\left(
T(f,\alpha_1,\alpha_2,q,\lambda)\cap A(R) 
\right)
\leq 4\pi m\,t(R)^2 + 8\pi R \,t(R).
\]
Theorem~\ref{theorem1} says that $\area T(f,\alpha_1,\alpha_2,q,\lambda)\cap A(R) \geq 3\pi\eta R^2$ and 
thus we obtain
\[
m\geq \frac{3\pi\eta R^2-8\pi R \,t(R)}{4\pi\, t(R)^2}\geq \frac{\eta}{2}\left(\frac{R}{t(R)}\right)^2
\]
for large~$R$.
\end{proof}
For $m=m(R)$ 
and $a_1(R),\dots,a_{m}(R)$ as in Lemma~\ref{disks}
we choose for each disk $D(a_j(R),t(R))$ a subset 
$V$ as in Lemma~\ref{applyahlfors}.
We denote these sets 
by $V_1(R),\dots,V_m(R)$.

It follows that
\begin{equation}\label{areaVj}
\area\left(\bigcup_{j=1}^{m(R)}
V_j(R)\right)\geq m \beta\, t(R)^2\geq \frac{\eta}{2}\beta R^2.
\end{equation}
Thus $\bigcup_{j=1}^{m(R)} V_j(R)$ has a positive density in $A(R)$.

We now construct the sets ${\mathcal E}_l$ to which Lemma~\ref{lemmamcm} will be applied.
We choose $R_0$ large and put 
\[
{\mathcal E}_0=\left\{A(R_0)\right\}
\quad\text{and}\quad
{\mathcal E}_1=\left\{V_j(R_0):1\leq j\leq m(R_0) \right\}.
\]
We shall define the sets ${\mathcal E}_l$ inductively such that if 
$F\in {\mathcal E}_l$, then $f^l(F)=A(R_{l,F})$ for some $R_{l,F}\geq R_0$. 
Moreover, 
$f^{l-1}: F\to f^{l-1}(F) $ is bijective and
if $G\in {\mathcal E}_{l-1}$ such that $F\subset G$, then 
\[
f^{l-1}(F)=V_j(R_{l-1,G})\subset D(a_j(R_{l-1,G}),t(R_{l-1,G}))
\]
for some $j\in\{1,\dots,m(R_{l-1,G})\}$.
To simplify notation, we will write $a_j$ instead of 
$a_j(R_{l-1,G})$ in the sequel.

Suppose now that ${\mathcal E}_l$ has been defined and let $F\in {\mathcal E}_l$
and $G\in {\mathcal E}_{l-1}$  be as above. 
By Lemma~\ref{applyahlfors} the disk $D(a_j,t(R_{l-1,G}))$
has a subdomain $U$ which is mapped by $\log f$ bijectively onto 
$\Omega_\nu(a_j)$ for some $\nu\in\{1,2,3\}$, with
$f^{l-1}(F)=V_j(R_{l-1,G})$  being the subset that is mapped onto
$Q_\nu(a_j)$.
Thus $\log f^l:F\to Q_\nu(a_j)$ is bijective and its
inverse $\psi:Q_\nu(a_j)\to F$ extends  to 
$\Omega_\nu(a_j)$.
For $1\leq k\leq m(R_{l,F})$ we can choose 
a domain $W_k\subset Q_\nu(a_j)$ such that $\exp W_k= V_k(R_{l,F})$.
We now put 
\[
{\mathcal E}_{l+1}(F)=\left\{\psi(W_k):1\leq k\leq m(R_{l,F})\right\}.
\]
Finally we set
\[
{\mathcal E}_{l+1}=\bigcup_{F\in {\mathcal E}_l} {\mathcal E}_{l+1}(F).
\]
Then the sequence $({\mathcal E}_l)$ has the desired properties.
Again we denote by $E_l$ the union of all elements of ${\mathcal E}_l$.
\begin{lemma} \label{densdiam}
There exists  $\Delta>0$ such that
if $F\in {\mathcal E}_l$, then
\begin{equation}\label{dens}
\dens\left(E_{l+1},F\right)\geq\Delta
\end{equation}
for all~$l$. Moreover,
\begin{equation}\label{diam}
\diam F \leq \exp\left(-l^2\right)
\end{equation}
for large~$l$.
\end{lemma}

\begin{proof}
Let $F\in {\mathcal E}_l$ and let $\psi$ and $W_k$ be as above.
Since $W_k=\log V_k(R_{l,F})$ for some branch of the
logarithm we have
\[
\area W_k =\int_{V_k(R_{l,F})}\frac{1}{|z|^2} \,dx\,dy
\geq \frac{1}{4R_{l,F}^2}\area V_k(R_{l,F})
\]
so that~\eqref{areaVj} yields
\[
\area\left(\bigcup_{k=1}^{m(R_{l,F})} W_k\right)
\geq \frac{\eta\beta}{8}.
\]
Applying Lemma~\ref{koebe1} we see that there exists a constant
$C$ such that $|\psi'(\zeta)|\leq C|\psi'(z)|$ for $z,\zeta\in Q_\nu(a_j)$.
(As the domains $\Omega_\nu(a_j)$  and the compact 
subsets $Q_\nu(a_j)$  are translates of fixed sets,
the constant $C$ does not depend on~$a_j$ or~$\nu$.)
Thus 
\[
\begin{aligned}
\dens(E_{l+1},F)
&=
\dens\left(\bigcup_{k=1}^{m(R_{l,F})} \psi(W_k), \psi(Q_\nu(a_j))\right)
\\ &
\geq 
\frac{1}{C^2} \dens\left(\bigcup_{k=1}^{m(R_{l,F})} W_k, Q_\nu(a_j)\right)
\\ &
\geq  \frac{\eta\beta}{32C^2\pi\log 2}.
\end{aligned}
\]
Thus~\eqref{dens} holds with $\Delta=\eta/(32C^2\pi\log 2)$.

To prove~\eqref{diam} let
$F_k\in {\mathcal E}_k$ such that $F\subset F_k$, for $1\leq k<l$.
(With $G$ as before we thus have $G=F_{l-1}$.) 
With the abbreviation $R_k=R_{k,F_k}$ we have
$f^k(F_k)=A(R_{k})$.
It follows from the construction and~\eqref{condition2} that $R_{k+1}\geq R_k^q$
and thus $R_k\geq (R_0)^{q^k}$.

As before we have 
\[
f^{l-1}(F)=V_j(R_{l-1})\subset D(a_j, t(R_{l-1}))
\subset A(R_{l-1})
\]
for some $j\in\{1,\dots,m(R_{l-1})\}$. Let 
$\phi$ be the branch of the inverse of $f^{l-1}$ which
maps $f^{l-1}(F)$ to~$F$.
Noting that $\phi$ is univalent in $D(a_j,2 t(R_{l-1}))$ 
we deduce from
Koebe's distortion theorem 
(i.e., Lemma~\ref{koebe2}) that if $z\in D(a_j, t(R_{l-1}))$,
then $|\phi'(z)| \leq 12 |\phi'(a_j)|$.
We conclude that
\[
\diam F\leq 12 \left|\phi'(a_j)\right| \diam f^{l-1}(F)\leq 24 \left|\phi'(a_j)\right| t(R_{l-1}).
\]

Now
$(f^{l-1})'(z)=\prod_{k=0}^{l-2} f'(f^k(z))$.
Since $f^k(z)\in A(R_k)$ it follows from~\eqref{gamma} that
\[
\left|f'(f^k(z))\right|\geq 
\gamma \frac{\left|f^{k+1}(z)\right|}{t(R_k)}\geq \gamma \frac{R_{k+1}}{t(R_k)}
=\delta \frac{R_{k+1}}{R_k}\log M(R_k,f)
\]
where $\delta = \gamma\alpha_1/(2\mu)$.
We conclude that 
\[
\left|(f^{l-1})'(z)\right|\geq \frac{R_{l-1}}{R_0}\prod_{k=0}^{l-2} \delta \log M(R_k,f).
\]
Thus 
\[
|\phi'(a_j)|\leq  \frac{R_0}{R_{l-1}}\prod_{k=0}^{l-2} \frac{1}{\delta\log M(R_k,f)}
\]
and hence
\[
\diam F \leq 24 \frac{R_0}{R_{l-1}}\prod_{k=0}^{l-2} \frac{1}{\delta\log M(R_k,f)}
t(R_{l-1})
=\tau \prod_{k=0}^{l-1} \frac{1}{\delta\log M(R_k,f)}
\]
with $\tau=48 R_0 \mu\delta/\alpha_1$.
For large $R_0$ we have $\delta\log M(r,f)\geq \log r$ if $r\geq R_0$.
Thus $\delta\log M(R_k,f)\geq \log R_k\geq q^k\log R_0\geq q^k$ if $R_0$ is chosen large enough. 
Hence
\[
\diam F \leq \tau \prod_{k=0}^{l-1} q^{-k}
=\tau\exp\left(-\tfrac{1}{2}(l-1)l\log q\right)
\leq \exp\left(-l^2\right)
\]
for large~$l$, since we have chosen $q=8>e^2$.
\end{proof}
Lemma~\ref{densdiam} says that
we can apply Lemma~\ref{lemmamcm} with 
\[
\Delta_l=\Delta
\quad\text{and}\quad
d_l=
\exp\left(-l^2\right)
.
\]
This yields $\dim E=2$.
Moreover, it follows from the construction that $E\subset I(f)$.
By Lemma~\ref{lemmars} we have $A(R)\cap J(f)\neq\emptyset$ for large~$R$.
This implies that $F\cap J(f)\neq\emptyset$ if $F\in {\mathcal E}_l$ and if $l$ is
sufficiently large. Hence $E\subset J(f)$.
Altogether we thus have $E\subset I(f)\cap J(f)$.
This completes the proof of Theorem~\ref{theorem2}.


\begin{thebibliography}{99}
\bibitem{Aspenberg09}
M.\ Aspenberg and W.\ Bergweiler,
Entire functions with Julia sets of positive measure.
Preprint, arxiv: 0904.1295. 
\bibitem{Baranski08}
K.\ Bara\'nski, Hausdorff dimension of hairs and ends for entire maps
of finite order.  Math. Proc. Cambridge Philos. Soc.
145 (2008), 719--737.
\bibitem{Baranski08a}
K.\ Bara\'nski, B.\ Karpi\'nska and A.\ Zdunik, Hyperbolic dimension
of Julia sets of meromorphic maps with logarithmic tracts.
Int.\ Math.\ Res.\ Not.\ 2009 (2009), 615--624.
\bibitem{Bergweiler93}
W.\ Bergweiler,
Iteration of meromorphic functions.
{\rm Bull.\ Amer.\ Math.\ Soc.\ (N.\ S.)}
29 (1993), 151--188.
\bibitem{Bergweiler08a}
W.\ Bergweiler, B. Karpi\'nska and G.\ M.\  Stallard, The growth
rate of an entire function and the Hausdorff dimension of its
Julia set. 
To appear in J.\ London Math.\ Soc., arXiv: 0807.2363.
\bibitem{Eremenko92}
A.\ E.\ Eremenko and M.\ Yu.\ Lyubich, Dynamical properties of some
classes of entire functions. {\rm Ann.\ Inst.\ Fourier} 42 (1992),
989--1020.
\bibitem{Fatou26} P.\ Fatou, 
Sur l'it\'eration des fonctions transcendantes enti\`eres.
{\rm Acta Math.} 47 (1926), 337-360.
\bibitem{Fuchs40}
A. J. Macintyre and W. H. J. Fuchs, 
Inequalities for the logarithmic derivatives of a polynomial.
J. London Math. Soc. 15 (1940), 162--168. 
\bibitem{Goldberg08}
A.\ A.\ Goldberg and I.\ V.\ Ostrovskii,
Value distribution of meromorphic functions.
Transl.\ Math.\ Monographs 236, American Math.\ Soc., Providence, R.~I., 2008.
\bibitem{Hayman64}
W.\ K.\ Hayman,
{\rm Meromorphic functions}.
Clarendon Press, Oxford, 1964.
\bibitem{Hayman65}
W.\ K.\ Hayman,
On the characteristic of functions meromorphic 
in the plane and of their integrals.
Proc.\ London Math.\ Soc.\ (3) 14a (1965), 93--128. 
\bibitem{Levin64}
B. Ja. Levin, 
Distribution of zeros of entire functions. 
American Mathematical Society, Providence, R.~I., 1964
\bibitem{McMullen87}
C.\ McMullen, Area and Hausdorff dimension of Julia sets of entire
functions. Trans.\ Amer.\ Math.\ Soc.\ 300 (1987), 329--342.
\bibitem{Miles96}
J.\ Miles and  J.\ Rossi, 
Linear combinations of logarithmic derivatives of entire functions 
with applications to differential equations.
Pacific J. Math. 174 (1996), 195--214. 
\bibitem{Rempe}
L.\ Rempe,
Rigidity of escaping dynamics for transcendental entire functions.
To appear in
Acta Math.,
arXiv: math/0605058.
\bibitem{RRRS}
G.\ Rottenfu{\ss}er, J. R\"uckert, L. Rempe and D.\ Schleicher,
Dynamic rays of bounded-type entire functions. 
To appear in 
Ann.\ of Math.,
arXiv: 0704.3213.
\bibitem{Schubert07}
H.\ Schubert, \"Uber die Hausdorff-Dimension der Juliamenge von
Funktionen endlicher Ordnung. Dissertation, University of Kiel,
2007; http://eldiss.uni-kiel.de/macau/receive/ 
\mbox{dissertation\_diss\_00002124}.
\bibitem{Stallard96} G.\ M.\  Stallard,
The Hausdorff dimension of Julia sets of entire functions II.
{\rm Math.\ Proc.\ Cambridge Philos.\ Soc.} 119 (1996), 513--536.
\bibitem{Stallard08}
G.\ M.\ Stallard,
Dimensions of Julia sets of transcendental meromorphic functions,
in ``Transcendental Dynamics and Complex Analysis''.
London Math.\ Soc.\ Lect.\ Note Ser.\ 348.
Edited by P.\ J.\ Rippon
and G.\ M.\ Stallard,
Cambridge Univ.\ Press, Cambridge, 2008, pp.~425--446.
\bibitem{Taniguchi03}
M.\ Taniguchi,
Size of the Julia set
of structurally finite transcendental entire function.
Math. Proc. Cambridge Philos. Soc. 135 (2003), 181--192.
\bibitem{Townsend87} 
D.~Townsend, 
Comparisons between $T(r,f)$ and the total variation of 
${\rm arg}\,f(re\sp {i\theta})$ and ${\rm log}\,\vert f(re\sp {i\theta})\vert $.
J.\ Math.\ Anal.\ Appl.\ 128 (1987), 347--361. 
\bibitem{Valiron23} G.\ Valiron, 
{\rm Lectures on the general theory
of integral functions}, 
\'Edouard Privat, Toulouse, 1923; Chelsea, New York, 1949.

\bibitem{Zheng06} J.-H. Zheng, On multiply-connected Fatou components
in iteration of meromorphic functions. 
J.~Math. Anal. Appl.
313 (2006), 24--37.

\end{thebibliography}
\end{document}